\documentclass[11pt]{amsart}

\usepackage{amssymb, amsmath, amsthm, mathrsfs, amsfonts, latexsym, color}
\usepackage[margin=1.2in]{geometry}

\newtheorem{theorem}{Theorem}[section]
\newtheorem{proposition}[theorem]{Proposition}

\newtheorem{remark}[theorem]{Remark}

\def\R{{\mathbb R}}

\def\E{{\mathbb E}}
\def\P{{\mathbb P}}

\newcommand{\1}{\mathbf{1}}

\numberwithin{equation}{section}

\begin{document}

\title[Local Nondeterminism and Modulus of Continuity for Stochastic Wave Equation]
{Local Nondeterminism and the Exact Modulus of Continuity for Stochastic Wave Equation}
\author{Cheuk Yin Lee
\and Yimin Xiao
}



\begin{abstract}
We consider the linear stochastic wave equation driven by a Gaussian noise.
We show that the solution satisfies a certain form of strong local nondeterminism
and we use this property to derive the exact uniform modulus of continuity for the
solution.
\end{abstract}

\maketitle

\section{Introduction}

Let $k \ge 1$ and $\beta \in (0, k \wedge 2)$, or $k = 1 = \beta$. We consider
the linear stochastic wave equation
\begin{align}\label{eq1.1}
\begin{cases}
\displaystyle{\frac{\partial^2}{\partial t^2}u(t, x) = \Delta u(t, x) + \dot{W}(t, x)},
\quad t \ge 0, x \in \mathbb{R}^k,\\
u(0, x) = \displaystyle{\frac{\partial}{\partial t}u(0, x) = 0}.
\end{cases}
\end{align}
Here, $\dot{W}$ is the space-time Gaussian white noise if $k = 1 =\beta$; and
is a Gaussian noise that is white in time and has a spatially homogeneous
covariance given by the Riesz kernel with exponent $\beta$ if $k \ge 1$ and
$\beta \in (0, k\wedge 2)$, i.e.
\[ \E(\dot{W}(t, x) \dot{W}(s, y)) = \delta(t-s) |x-y|^{-\beta}. \]
The existence of real-valued process solution to \eqref{eq1.1} was discussed
in \cite{Walsh,D99}. Regarding the sample paths of the solution, results on
the H\"{o}lder regularity and hitting probability have been proved in \cite{DS10}.
In this present paper, we determine the exact uniform modulus of continuity
of the solution $u(t, x)$ in the time and space variables $(t, x)$. For this purpose,
we show that the Gaussian random field $\{u(t, x), t \ge 0, x \in \R^k\}$ satisfies
a form of strong local nondeterminism.

The property of local nondeterminism is useful for investigating sample paths
of Gaussian random fields. This notion was first introduced by Berman \cite{B73}
for Gaussian processes and extended by Pitt \cite{P98} for Gaussian random fields
to study their local times. Later, the property of strong local nondeterminism
was developed to study exact regularity of local times, small ball probability
 and other sample paths properties for Gaussian random fields (see, e.g., \cite{X06,X08}).

It is well known that the Brownian sheet does not satisfy the property of (strong) local
nondeterminism (in the sense of Pitt \cite{P98}) but it satisfies sectorial local nondeterminism
\cite[Proposition 4.2]{KX07}. Recall from \cite[Theorem 3.1]{Walsh} that when $k = 1
= \beta$ and $\dot{W}$ is the space-time white noise, the solution $u(t, x)$ of \eqref{eq1.1}
has the representation
\begin{equation}\label{beta1}
u(t, x) = \frac{1}{2}\,\hat{W}\left(\frac{t-x}{\sqrt{2}}, \frac{t+x}{\sqrt{2}}\right),
\end{equation}
where $\hat{W}$ is a modified Brownian sheet (cf. \cite[p.281]{Walsh}). In this
case, many properties of the solution $u(t, x)$ can be derived from those of
$\hat{W} (t,x).$ For $\beta \ne 1$ or $k \ge 2$, there are few precise results (such
as the exact modulus of continuity, modulus of non-differentiability, multifractal
analysis of exceptional oscillations) for the sample function $u(t, x)$.
Investigation of these problems naturally leads to the study of local nondeterminism
for the solution $u(t, x)$.

In this paper, we investigate the property of local nondeterminism for the solution
of \eqref{eq1.1} and use this property to study the uniform modulus of continuity
of its sample functions. The main results of this paper are Proposition \ref{prop2.1}
and Theorem \ref{Th:mod}. Proposition \ref{prop2.1} shows that for a general
dimension $k$, the solution $u(t, x)$  satisfies an integral form of local
nondeterminism. When $k = 1$ and $\beta = 1$, this property (see \eqref{beta2}
below) can also be derived from the sectorial local nondeterminism for the Brownian
sheet in \cite[Proposition 4.2]{KX07} after a change of coordinates.  While for
$k = 1$ and $\beta \in (0, 1)$, the property \eqref{beta2} is similar to the sectorial
local nondeterminism in \cite[Theorem 1]{WX07} for a fractional Brownian sheet,
which suggests that the sample function $u(t, x)$ may have some subtle properties
that are different from those of Gaussian random fields with stationary increments
(an important example of the latter is fractional Brownian motion). We believe that
Proposition \ref{prop2.1} is useful for studying precise regularity and other
sample path properties of $u(t, x)$. In Theorem \ref{Th:mod}, we use it to derive
the exact uniform modulus of continuity of $u(t, x)$.

\medskip

{\bf Acknowledgements.} The authors thank Professor Raluca Balan and Ciprian Tudor
for stimulating discussions and for their generosity in encouraging the authors
to publish this paper. The research of Yimin Xiao
is partially supported by NSF grants DMS-1607089 and DMS-1855185.

\section{Local Nondeterminism}

Let $G$ be the fundamental solution of the wave equation.
Recall that if $k=1$, $G(t, x) = \frac{1}{2}\1_{\{ |x| < t \}}$; if
$k \ge 2$ and $k$ is even,
\begin{align*}
G(t, x) &= c_k \left(\frac{1}{t} \frac{\partial}{\partial t}\right)^{
(k-2)/2} (t^2 - |x|^2)^{-1/2}_+;
\end{align*}
if $k \ge 3$ and $k$ is odd,
\begin{align*}
G(t, x) &= c_k \left( \frac{1}{t} \frac{\partial}{\partial t} \right)^{(k-3)/2}
\frac{\sigma^k_t(dx)}{t},
\end{align*}
where $\sigma^k_t$ is the uniform surface measure on the sphere
$\{ x \in \mathbb{R}^k : |x| = t \}$, see \cite[Chapter 5]{Folland}.
Note that for $k \ge 3$, $G$ is not a function but a distribution.
Also recall that for any dimension $k \ge 1$, the Fourier transform
of $G$ in variable $x$ is given by
\begin{align}\label{eq2.1}
\mathscr{F}(G(t, \cdot))(\xi) = \frac{\sin(t|\xi|)}{|\xi|}, \quad t \ge 0,
\xi \in \mathbb{R}^k.
\end{align}

In \cite{D99}, Dalang extended Walsh's stochastic integration and
proved that the real-valued process solution of equation
\eqref{eq1.1} is given by
\[
u(t, x) = \int_0^t\int_{\mathbb{R}^k} G(t-s, x-y)\, W(ds\,dy),
\]
where $W$ is the martingale measure induced by the noise $\dot{W}$.
The range of $\beta$ has been chosen so that the stochastic integral
exists. Recall from Theorem 2 of \cite{D99} that
\begin{align}\label{eq2.2}
\E\bigg[\bigg(\int_0^t \int_{\mathbb{R}^k} H(s, y) W(ds\, dy) \bigg)^2\bigg]
= c_{k, \beta} \int_0^t ds \int_{\mathbb{R}^k} d\xi \, |\xi|^{\beta-k}
|\mathscr{F}(H(s, \cdot))(\xi)|^2
\end{align}
provided that $s \mapsto H(s, \cdot)$ is a deterministic function with
values in the space of nonnegative distributions with rapid decrease and
\[\int_0^t ds \int_{\mathbb{R}^k} d\xi \, |\xi|^{k-\beta}
|\mathscr{F}(H(s, \cdot)(\xi)|^2 < \infty.\]
The following result shows that the solution $u(t, x)$ satisfies
a certain form of strong local nondeterminism.

\begin{proposition}\label{prop2.1}
Let $0 < a < a' < \infty$ and $0 < b < \infty$. There exist
constants $C > 0$ and $\delta > 0$ depending on $a$, $a'$ and $b$
such that for all integers $n \ge 1$ and all $(t, x), (t^1, x^1), \dots,
(t^n, x^n)$ in $[a, a'] \times [-b, b]^k$ with $|t-t^j| + |x-x^j|\le \delta$,
we have
\begin{equation}\label{Eq:SLND}
 \operatorname{Var}{(u(t, x)| u(t^1, x^1), \dots, u(t^n, x^n))}
\ge C \int_{\mathbb{S}^{k-1}} \min_{1\le j \le n} |(t-t^j) + (x - x^j)
\cdot w|^{2-\beta}\, dw,
\end{equation}
where $dw$ is the surface measure on the unit sphere $\mathbb{S}^{k-1}$.
\end{proposition}

\begin{remark}
When $k = 1$, the surface measure $dw$ in \eqref{Eq:SLND}
is supported on $\{-1, 1\}$.  It follows that $u(t, x)$ satisfies
sectorial local nondeterminism:
\begin{equation}\label{beta2}
\begin{split}
& \operatorname{Var}(u(t, x)| u(t^1, x^1), \dots, u(t^n, x^n))\\
& \ge C \left( \min_{1\le j \le n}|(t-t^j) + (x - x^j)|^{2-\beta}
+ \min_{1\le j \le n}|(t-t^j) - (x - x^j)|^{2-\beta} \right).
\end{split}
\end{equation}
When $\beta = 1$, this can be derived from (\ref{beta1}) and
Proposition 4.2 in \cite{KX07} by a change of coordinates
$(t, x) \mapsto (t+x, t-x)$. When $\beta \ne 1$, (\ref{beta2})
is similar to Theorem 1 in \cite{WX07} for a fractional Brownian
sheet, after the change of coordinates.\footnote{Professor Ciprian Tudor
showed us that the relation (\ref{beta1}) still holds if $\hat{W}$
is replaced by an appropriate Gaussian random field related to a
fractional Brownian sheet. This connection provides an explanation for
the similarity between (\ref{beta2}) and Theorem 1 in \cite{WX07}.}
We remark that (\ref{beta2})
is different from the strong local nondeterminism for
Gaussian random fields with stationary increments in \cite{LuanX12}.
This suggests that the solution process $u(t,x)$ may have some
subtle properties that are different from those of Gaussian random
fields with stationary increments such as a fractional Brownian motion.
\end{remark}

\begin{proof}[Proof of Proposition \ref{prop2.1}]
Take $\delta = a/2$. For each $w \in \mathbb{S}^{k-1}$, let
\[
r(w) = \min_{1\le j \le n}|(t^j - t) - (x^j - x) \cdot w|.
\]
Since $u$ is a centered Gaussian random field, the conditional
variance $\operatorname{Var}(u(t, x)| u(t^1, x^1), \dots,$ $u(t^n, x^n))$
 is the squared distance of $u(t, x)$ from the
linear subspace spanned by $u(t^1, x^1), \dots,$ $u(t^n, x^n)$
in $L^2(\P)$. Thus, it suffices to show that there exist constants
$C > 0$ and $\delta > 0$  such that for all $(t, x), (t^1, x^1), \dots,
(t^n, x^n)$ in $[a, a'] \times [-b, b]^k$ with $|t-t^j| + |x-x^j|
\le \delta$, we have
\begin{equation}\label{eq2.3}
\E\bigg[ \bigg( u(t, x) - \sum_{j=1}^n \alpha_j u(t^j, x^j)
\bigg)^2 \bigg] \ge C \int_{\mathbb{S}^{k-1}} r(w)^{2-\beta}\,dw
\end{equation}
for any choice of real numbers $\alpha_1, \dots, \alpha_n$.
Using \eqref{eq2.1}, \eqref{eq2.2} and spherical coordinate $\xi
= \rho\, w$, we have
\begin{align*}
& \E\bigg[ \bigg( u(t, x) - \sum_{j=1}^n \alpha_j u(t^j, x^j) \bigg)^2 \bigg]\\
& = c_{k, \beta}\int_0^\infty ds \int_{\mathbb{R}^k} \frac{d\xi}{ |\xi|^{2+k - \beta }}
\bigg| \sin((t-s)|\xi|) \1_{[0, t]}(s) - \sum_{j=1}^n \alpha_j
e^{-i(x^j - x) \cdot \xi} \sin((t^j-s)|\xi|) \1_{[0, t^j]}(s)\bigg|^2\\
& \ge c_{k, \beta}\int_0^{a/2} ds \int_0^\infty \frac{ d\rho} { \rho^{3- \beta }}
\int_{\mathbb{S}^{k-1}} dw \bigg|
\sin((t-s)\rho) - \sum_{j=1}^n \alpha_j e^{-i\rho(x^j - x)\cdot w} \sin((t^j-s)\rho) \bigg|^2\\
& = \frac{c_{k, \beta}}{8} \int_0^{a/2} ds
\int_{-\infty}^\infty \frac{d\rho}{|\rho|^{3-\beta }} \int_{\mathbb{S}^{k-1}} dw
\bigg| \left(e^{i(t-s)\rho} - e^{-i(t-s)\rho}\right) \\
& \qquad \qquad \qquad  \qquad \qquad \qquad  \qquad \qquad \quad
- \sum_{j=1}^n \alpha_j e^{-i\rho(x^j - x)\cdot w}
\left(e^{i(t^j-s)\rho} - e^{-i(t^j-s)\rho}\right)\bigg|^2 \\
& =: \frac{c_{k, \beta}}{8} \int_{\mathbb{S}^{k-1}} A(w)\, dw.
\end{align*}
Let $\lambda = \min\{1, a/[2(a'+2\sqrt{k}b)]\}$ and consider the bump function
$\varphi: \mathbb{R} \to \mathbb{R}$ defined by
\[
\varphi(y) = \begin{cases}
\exp\left( 1 - \frac{1}{1-|\lambda^{-1}y|} \right), & |y| < \lambda,\\
0, & |y| \ge \lambda.
\end{cases} \]
Let $\varphi_r(y) = r^{-1}\varphi(y/r)$. For each $w \in \mathbb{S}^{k-1}$
such that $r(w) > 0$, consider the integral
\begin{align*}
I(w) := \int_0^{a/2} ds \int_{-\infty}^\infty d\rho \bigg[&\left(e^{i(t-s)\rho} - e^{-i(t-s)\rho}\right)\\
& - \sum_{j=1}^n \alpha_j e^{-i\rho(x^j - x)\cdot w} \left(e^{i(t^j-s)\rho} - e^{-i(t^j-s)\rho}\right) \bigg]
e^{-i(t-s)\rho} \widehat{\varphi}_{r(w)}(\rho).
\end{align*}
By the inverse Fourier transform (or one can apply the Plancherel theorem),
we have 
\begin{align*}
I(w) &= 2\pi \int_0^{a/2} ds \bigg[ \varphi_{r(w)}(0) - \varphi_{r(w)}\big(2(t-s)\big)\\
& \quad -\sum_{j=1}^n \alpha_j \Big(\varphi_{r(w)}\big((x^j-x)\cdot w - (t^j-t)\big) -
\varphi_{r(w)}\big((x^j-x)\cdot w - (t^j-t) + 2(t^j-s)\big) \Big)\bigg].
\end{align*}
Note that $r(w) \le |t^j - t| + |x^j - x| \le a' + 2\sqrt{k}b$.
For any $s \in [0, a/2]$, we have $2(t-s)/r(w) \ge a/[(a'+2\sqrt{k}b)]$
and $|(x^j-x)\cdot w - (t^j-t)|/r(w) \ge 1$, thus
\[
\varphi_{r(w)}\big(2(t-s)\big) =0 \ \hbox{ and }\
\varphi_{r(w)}\big((x^j-x)\cdot w - (t^j-t)\big) = 0
\ \hbox{ for } j = 1, \dots, n.
\]
Also, $[(x^j-x)\cdot w - (t^j-t) + 2(t^j-s)]/r(w)
\ge (-\delta+a)/[(a'+2\sqrt{k}b)] \ge \lambda$, thus
\[\varphi_{r(w)}\big((x^j-x)\cdot w - (t^j-t) + 2(t^j-s)\big) = 0.\]
It follows that 
\[
I(w) =  a\pi\, r(w)^{-1}.
\]
On the other hand, by the Cauchy--Schwarz inequality and scaling,
we obtain
\begin{align*}
(a\pi)^2 r(w)^{-2} = |I(w)|^2 &\le A(w)\times \int_0^{a/2} ds
\int_{-\infty}^\infty d\rho \, |\widehat{\varphi}(r(w) \rho)|^2|\rho|^{3-\beta}\\
& = (a/2) A(w) r(w)^{\beta - 4}\int_{-\infty}^\infty d\rho \,
|\widehat{\varphi}(\rho)|^2 |\rho|^{3-\beta}\\
& = C A(w) r(w)^{\beta - 4}
\end{align*}
for some finite constant $C$. Hence we have
\begin{equation} \label{Eq:A}
 A(w) \ge C'r(w)^{2-\beta}
 \end{equation}
and this remains true if $r(w) = 0$. Integrating both sides
of (\ref{Eq:A}) over $\mathbb{S}^{k-1}$ yields \eqref{eq2.3}.
\end{proof}

\section{Exact Uniform Modulus of Continuity}

It is known that sectorial local nondeterminism is useful
for proving the exact uniform modulus of continuity for
Gaussian random fields \cite{MWX13}. In this section we show
that the form of local nondeterminism in Proposition \ref{prop2.1}
can serve the same purpose for deriving the exact uniform
modulus of continuity of $u(t, x)$.

Let us denote
$$\sigma\big[(t, x), (t', x')\big] = \E[(u(t,x) - u(t', x'))^2]^{1/2}.$$
Recall from \cite[Proposition 4.1]{DS10} that for any $0 < a < a'
 < \infty$ and $0 < b < \infty$, there are positive constants $C_1$
 and $C_2$ such that
\begin{equation}\label{eq10.1}
C_1\bigg(|t-t'|+\sum_{j=1}^k|x_j-x'_j|\bigg)^{2-\beta}
\le \sigma[(t, x), (t', x')]^2 \le C_2 \bigg(|t-t'| +
\sum_{j=1}^k|x_j-x'_j|\bigg)^{2-\beta}
\end{equation}
for all $(t, x), (t', x') \in [a, a']\times [-b, b]^{k}$.

The following result establishes the exact uniform modulus of continuity
of $u(t, x)$ in the time and space variables $(t, x)$.

\begin{theorem}\label{Th:mod}
Let $I = [a, a'] \times [-b, b]^{k}$, where $0 < a < a' < \infty$
and $0 < b < \infty$. Let
\[
\gamma\big[(t, x), (t', x')\big] = \sigma\big[(t, x), (t', x')\big]
\sqrt{\log{(1+\sigma\big[(t, x), (t', x')\big]^{-1})}}.
\]
Then there is a positive finite constant $K$ such that
\begin{equation}\label{Eq:UM}
\lim_{\varepsilon \to 0+} \sup_{\substack{(t, x), (t', x') \in I,
\\ \sigma [(t, x), (t', x')] \le \varepsilon}} \frac{|u(t, x) - u(t', x')|}
{\gamma\big[(t, x),(t', x')\big]} = K, \quad \hbox{\rm a.s.}
\end{equation}
\end{theorem}

\begin{proof}
For any $\varepsilon > 0$, let
\[
J(\varepsilon) = \sup_{\substack{(t, x), (t', x') \in I,\\
\sigma[(t, x), (t', x')] \le \varepsilon}} \frac{|u(t, x) - u(t', x')|}
{\gamma\big[(t, x),(t', x')\big]}. \]
Since $\varepsilon \mapsto J(\varepsilon)$ is non-decreasing,
we see that the limit $\lim_{\varepsilon \to 0+} J(\varepsilon)$ 
exists a.s. In order to prove (\ref{Eq:UM}), we prove the following 
statements: there exist positive and finite constants $K^*$ and 
$ K_*$ such that 
\begin{equation}\label{Eq:UmU}
\lim_{\varepsilon \to 0+} J(\varepsilon)
\le K^*, \quad \hbox{ \rm  a.s.}
\end{equation} 
 and 
\begin{equation}\label{Eq:UmL}
\lim_{\varepsilon \to 0+} J(\varepsilon)
\ge K_*, \quad \hbox{ \rm  a.s.}
\end{equation} 
Then the conclusion of Theorem \ref{Th:mod} follows from Lemma 7.1.1 
of \cite{MR} where $\tau$ is chosen to be the Euclidean metric and $d$ is the 
canonical metric $\sigma[(t, x), (t', x')]$. [It is a 0-1 law for the modulus of continuity 
which is obtained by applying Kolmogorov's 0-1 law to the Karhunen--Lo\`eve expansion 
of $u(t,x)$.]


The proof of the upper bound \eqref{Eq:UmU} is standard. 
For any $\varepsilon > 0$, denote by $N(I, \varepsilon,\sigma)$  the 
smallest number of balls of radius $\varepsilon$ in the canonical metric $\sigma\big[(t, x), (t', x')\big]$ 
that are needed to cover the compact interval $I$. By the 
upper bound in \eqref{eq10.1}, we have 
$N(I, \varepsilon,\sigma) \le C \varepsilon^{-(1+k)/(2-\beta)}.$
Hence (\ref{Eq:UmU}) follows from the metric entropy bound for the uniform
modulus of continuity of a Gaussian field (cf. e.g., \cite[Theorem 1.3.5]{AT07}
or \cite{MR}).

Next we prove the  lower bound (\ref{Eq:UmL}). This is accomplished by 
applying Proposition \ref{prop2.1}, a conditioning argument and the 
Borel--Cantelli lemma. We first choose $\delta$ according to
Proposition \ref{prop2.1} and let $\delta' = \min\{\delta/(1+\sqrt{k}), a'-a, 2b\}$.
Note that $\delta'$ depends only on $a$, $a'$ and $b$.
For each $n \ge 1$, let
\[
\varepsilon_n = [C_2((1+k)\delta')^{2-\beta}2^{-(2-\beta)n}]^{1/2}.
\]
For $i = 0, 1, \dots, 2^n$, let $t^{n, i} = a+ i\delta'2^{-n}$ and
$x^{n, i}_j = -b + i\delta'2^{-n}$.
Then
\begin{align*}
\lim_{\varepsilon \to 0+} J(\varepsilon) &= \lim_{n \to \infty}
\sup_{\substack{(t, x), (t', x') \in I,\\ \sigma[(t, x), (t', x')]
\le \varepsilon_n}} \frac{|u(t, x) - u(t', x')|}{\gamma[(t, x),(t', x')]}\\
& \ge \liminf_{n \to \infty} \max_{1 \le i \le 2^n}
\frac{|u(t^{n, i}, x^{n, i}) - u(t^{n, i-1}, x^{n, i-1})|}
{\varepsilon_n \sqrt{\log(1+\varepsilon_n^{-1})}}\\
&=: \liminf_{n \to \infty} J_n.
\end{align*}
To obtain the inequality, we have used the fact that $\sigma[(t^{n, i}, x^{n, i}), (t^{n, i-1}, x^{n, i-1})] \le \varepsilon_n$ and that the function 
$\varepsilon \mapsto \varepsilon \sqrt{\log (1/\varepsilon)}$ is 
increasing for $\varepsilon$ small. 

Let $K_* > 0$ be a constant whose value will be determined later.
Fix $n$ and write $t^{n, i} = t^i$, $x^{n, i} = x^i$ to simplify notations.
By conditioning, we can write
\begin{align}
\begin{aligned}
\label{eq10.2}
& \P\left( J_n \le K_* \right)\\
& = \P\bigg(\max_{1 \le i \le 2^n} \frac{|u(t^i, x^i) - u(t^{i-1}, x^{i-1})|}
{\varepsilon_n \sqrt{\log(1+\varepsilon_n^{-1})}} \le K_* \bigg)\\
& = \E\Bigg[ \1_{A} \P\Bigg( \frac{|u(t^{2^n}, x^{2^n}) - u(t^{2^n-1},
x^{2^n-1})|}{\varepsilon_n \sqrt{\log(1+\varepsilon_n^{-1})}}
\le K_*  \bigg| u(t^i, x^i) : 0 \le i \le 2^n-1 \Bigg) \Bigg],
\end{aligned}
\end{align}
where $A$ is the event defined by
\[
A = \Bigg\{\max_{1 \le i \le 2^n-1} \frac{|u(t^i, x^i) - u(t^{i-1}, x^{i-1})|}
{\varepsilon_n \sqrt{\log(1+\varepsilon_n^{-1})}} \le K_*\Bigg\}.
\]
Since $|t^{2^n}-t^i| + |x^{2^n} - x^i| \le \delta$, by Proposition
\ref{prop2.1} we have
\begin{equation}\label{Eq:Cvar}
\begin{split}
&\operatorname{Var}{\left(u(t^{2^n}, x^{2^n})| u(t^i, x^i) : 0 \le i \le 2^n-1\right)}\\
& \ge C \int_{\mathbb{S}^{k-1}} \min_{0 \le i \le 2^n-1} |(t^{2^n}-t^i) + (x^{2^n} - x^i) \cdot w|^{2-\beta} \, dw\\
& \ge C \int_{\{w \in \mathbb{S}^{k-1} : \, (1, \dots, 1) \cdot w \ge 0\}}  \min_{0 \le i \le 2^n-1}
|\delta'(2^n-i)2^{-n} + \delta'(2^n-i)2^{-n}(1, \dots, 1) \cdot w|^{2-\beta} \, dw\\
& \ge C(\delta')^{2-\beta} \,2^{-(2-\beta)n} \int_{\{w \in \mathbb{S}^{k-1} : \, (1, \dots, 1) \cdot w \ge 0\}}  dw\\
&= C_0\, \varepsilon_n^2
\end{split}
\end{equation}
for some constant $C_0 > 0$ depending on $a$, $a'$ and $b$.

Since the conditional distribution of $u(t^{2^n}, x^{2^n})$ is Gaussian with conditional variance 
$\operatorname{Var}{\left(u(t^{2^n}, x^{2^n})| u(t^i, x^i) : 0 \le i \le 2^n-1\right)}$, 
it follows from Anderson's inequality \cite{A55} and (\ref{Eq:Cvar}) that
\begin{align*}
& \P\Bigg( \frac{|u(t^{2^n}, x^{2^n}) - u(t^{2^n-1},
x^{2^n-1})|}{\varepsilon_n \sqrt{\log(1+\varepsilon_n^{-1})}}
\le K_*  \bigg| u(t^i, x^i) : 0 \le i \le 2^n-1 \Bigg)\\
& \le \P\Bigg( \frac{|u(t^{2^n}, x^{2^n})|}{\varepsilon_n \sqrt{\log(1+\varepsilon_n^{-1})}}
\le K_*  \bigg| u(t^i, x^i) : 0 \le i \le 2^n-1 \Bigg)\\
& \le \P\left( |Z| \le K_* \sqrt{C_0^{-1}
\log{(1+\varepsilon_n^{-1})}} \right)
\end{align*}
where $Z$ is a standard normal random variable. Using $\P(|Z| > x)
\ge (\sqrt{2\pi})^{-1} x^{-1} \exp(-x^2/2)$ for $x \ge 1$ and $1+\varepsilon^{-1} < 2/\varepsilon$ for $\varepsilon$ small, we deduce
that when $n$ is large the above probability is bounded from above by
\begin{align*}
1 - \frac{C(\varepsilon_n/2)^{K_*^2/(2C_0)}}{K_*\sqrt{\log{(2/\varepsilon_n)}}}
\le \exp\left( - \frac{C(\varepsilon_n/2)^{K_*^2/(2C_0)}}
{K_*\sqrt{\log{(2/\varepsilon_n)}}} \right) \le \exp\Bigg(- \frac{C_{K_*}2^{-\frac{(2-\beta)K_*^2}{4C_0}n}}{\sqrt{n}} \Bigg)
\end{align*}
where $C_{K_*}>0$ is a constant depending on $K_*$.
Then by \eqref{eq10.2} and induction, we have
\begin{align*}
\P \big(J_n \le K_*\big) \le \exp\Bigg( - 2^n \frac{C_{K_*}2^{-\frac{(2-\beta)K_*^2}{4C_0}n}}
{\sqrt{n}} \Bigg).
\end{align*}
We can now choose $K_* > 0$ to be a sufficiently small constant
such that $1-\frac{(2-\beta)K_*^2}{4C_0}>0$. Then 
$\sum_{n=1}^\infty \P\big(J_n \le K_*\big) < \infty$. Hence, by the Borel--Cantelli lemma,
$\liminf_n J_n \ge K_*$ a.s.\ and the proof is complete.
\end{proof}

\bigskip

\medskip


\textsc{Cheuk Yin Lee}: Department of Statistics
and Probability, C413 Wells Hall, Michigan State University, East Lansing, MI 48824, U.S.A.\\
E-mail: \texttt{leecheu1@msu.edu}\\
\medskip


\textsc{Yimin Xiao}: Department of Statistics
and Probability, C413 Wells Hall, Michigan State University, East Lansing, MI 48824, U.S.A.\\
E-mail: \texttt{xiao@stt.msu.edu}\\

\end{document}